\documentclass[11pt]{article}
\usepackage[a4paper,textwidth=15.5cm,textheight=21cm,includehead]{geometry}

\usepackage[spanish,english]{babel}
\usepackage[latin1]{inputenc}

\usepackage{latexsym}
\usepackage{amsthm}
\usepackage{epic}
\usepackage{amsfonts}

\usepackage{rotating}

\usepackage[T1]{fontenc}

\usepackage{amsmath}  
\usepackage{amssymb}  
\usepackage[all]{xy}   
\usepackage{graphicx}   
\usepackage[spanish]{babel}   
\usepackage{mathrsfs}
\usepackage{euscript}

\newtheorem{thm}{Theorem}[section]
\newtheorem{prop}[thm]{Proposition}
\newtheorem{bla}[thm]{Theorem}

\newtheorem{lem}[thm]{Lemma}

\newcommand{\C}{\mathbb{C}}

\newcommand{\Pe}{\mathbb{P}}

\begin{document}

\title{\textbf{THE HYPERDETERMINANT VANISHES FOR ALL BUT TWO SCHUR FUNCTORS}}

\author{Alicia Tocino Sánchez }
\maketitle

\begin{abstract}
We recall the notion of hyperdeterminant of a multidimensional matrix (tensor). We prove that if we restrict the hyperdeterminant to
a skew-symmetric tensor $\bigwedge^p V\subseteq V^{\otimes p}$ with $p\geq 3$ then it vanishes. 
The hyperdeterminant also vanishes when we restrict it to the space $\Gamma^\lambda\otimes S_\lambda V\subset V^{\otimes p}$
where $\lambda$ is a Young diagram with $p$ boxes and $\lambda_2\ge 2$ or $\lambda_3\geq 1$. 
\end{abstract}

\section{Introduction}

	In \cite{CA} A. Cayley described for the first time the notion of hyperdeterminant of a multidimensional matrix (tensor) 
	which is analogous to the determinant of a square matrix. In \cite{OED} L. Oeding computed the hyperdeterminant for symmetric
	tensors, ${\mathrm Sym}^p V\subseteq V^{\otimes^p}$, describing all its irreducible factors. G. Ottaviani proposed in \cite{OTT} to compute 
	the irreducible factors of the hyperdeterminant of a skew-symmetric tensor  
	$\bigwedge^p V\subseteq V^{{\otimes p}}$. In this paper we give a solution not only for skew-symmetric tensors but also for 
	any Schur functor $\Gamma^\lambda V\otimes S_\lambda V\subseteq V^{\otimes p}$ where $\lambda$ is a Young diagram 
	with $|\lambda|=p$ and $\lambda_2\ge 2$ or $\lambda_3\geq 1$.
	
	\textbf{Main Theorem.} \textit{When $A\in\Gamma^{\lambda} V\otimes S_{\lambda}V\subseteq V^{\otimes p}$, $|\lambda|=p$ and $p\ge 2$, 
	$Det(A)$ can be nonzero only for $\lambda=(p)$ 
	(corresponding to the symmetric power ${\mathrm Sym}^p V$) and $\lambda=(p-1,1)$, where $\Gamma^{\lambda}V$ 
	(resp. $S_{\lambda}V$) is the  $\Sigma_p$-module (resp. $GL(V)-$module) associated to $\lambda$.}

	The case $p=2$ of \textit{Main Theorem} corresponds to skew symmetric matrices in $\bigwedge^2V$.
	In this case the determinant restricts to the square of the Pfaffian.

\section{Notations and preliminaries}


	Let $V$ be a complex vector space of dimension $n$ and let $\Sigma_p$ be the symmetric group of permutations over $p$ elements. We define
	
	\textit{tensor product:} $V^{\otimes ^p}=$ {\small $\{A:\underbrace{V^\vee\times\ldots\times V^\vee}_{p}\to \C$ multilinear $\}$ }
	
	\textit{skew-symmetric	tensor product:} $\bigwedge^{p} V=$
	\\{\small $\{A:\underbrace{V^\vee\times \ldots\times V^\vee}_{p}\to \C$ multilinear 
	s.t. $A(x_1,\ldots,x_p)=sign(\sigma)A(x_{\sigma(1)},\ldots,x_{\sigma(p)})$, $\forall\sigma\in\Sigma_p\}$}
			
	\textit{symmetric tensor product:} ${\mathrm Sym}^p V=$
	\\{\small $\{A:\underbrace{V^\vee\times \ldots\times V^\vee}_{p}\to \C$ 
	multilinear s.t. $A(x_1,\ldots,x_p)=A(x_{\sigma(1)},\ldots,x_{\sigma(p)})$, $\forall\sigma\in\Sigma_p\}$}
		
	We recall the definition of Schur functor, see \cite{FH}.	A \textit{Young diagram} denoted by
	$\lambda=(\lambda_1,\lambda_2,\ldots,\lambda_m)$, where $\lambda_1\geq\lambda_2\geq\ldots\geq\lambda_m\geq 0$ 
	consists of a collection of boxes ordered in 
	consecutive rows, where the $i-$th row has exactly $\lambda_i$ boxes. The number of boxes of $\lambda$ is denoted by 
	$|\lambda|=\lambda_1+\lambda_2+\ldots+\lambda_m$. 
	We also denote by $\lambda^*$ the conjugate Young diagram of $\lambda$.
			 
		 
	Any filling of $\lambda$ with numbers is called a \textit{tableau}. Just to fix convention, for a given Young diagram, number 
	the boxes consecutively  (from left to right and top to bottom). Here we use all numbers from $1$ to $p$ in order to fill $p$ boxes. 
	More generally, a tableau can allow repetitions
	of numbers. Each filling describes a vector in $V^{\otimes p}$.
		 
	Let $\Sigma_p$ be the symmetric group of permutations over $p$ elements. Due to the filling, we can consider the elements of $\Sigma_p$
	as permuting the boxes. Let
	$$P=P_{\lambda}=\{\sigma\in\Sigma_p\,:\,\sigma\,\,preserves\,\,each\,\,row\}$$
	$$Q=Q_\lambda=\{\sigma\in\Sigma_p\,:\,\sigma\,\,preserves\,\,each\,\,column \}$$
	They depend on $\lambda$ but also on the filling of $\lambda$.
			
	In the group algebra $\C\Sigma_p$, we introduce two elements corresponding to these subgroups
	
	$$a_\lambda=\sum_{\sigma\in P}e_\sigma$$  $$b_\lambda=\sum_{\sigma\in Q}sign(\sigma)\cdot e_\sigma$$
	
	As $\Sigma_p$ acts on the $p-$th tensor power $V^{\otimes p}$ by permuting factors, then
	$$Im(a_\lambda)=S^{\lambda_1}V\otimes\ldots\otimes S^{\lambda_m} V\subseteq V^{\otimes p}$$
	$$Im(b_\lambda)=\bigwedge^{\mu_1}V\otimes\ldots\otimes\bigwedge^{\mu_l}V\subseteq V^{\otimes p}$$
	where $\mu$ is the conjugate partition of $\lambda$.

	Finally, we set  $c_\lambda=a_\lambda\cdot b_\lambda\in\C\Sigma_p $ that
	is called the \textit{Young symmetrizer} corresponding to $\lambda$. We denote the image of 
	$c_\lambda$ on $\C\Sigma_p$ by $\Gamma^\lambda V$ , it is a irreducible $\Sigma_p$-module.
	$$\Gamma^\lambda V=Im(c_\lambda|_{\C\Sigma_p})$$
	We call the functor that associates $V\leadsto \Gamma^\lambda V$ the \textit{Schur functor}. 
	We also denote by $S_\lambda V$ the image of $c_\lambda$ acting on $V^{\otimes p}$ 
	which is an irreducible GL(V)-module and is non zero if and only if
	the number of rows is less or equal than $n$. With these definitions we can state the Schur-Weyl duality 
	(see \cite{PRO}): 
	
	\textit{There is a $\Sigma_p \times SL(V)$-decomposition $V^{\otimes ^p}=\oplus_\lambda \Gamma^\lambda V\otimes S_{\lambda} V$
	where the sum is extended on all Young diagrams with $p$ boxes}

	With the purpose of giving the equations for $\Gamma^\lambda V\otimes S_\lambda V$ where $\lambda$ is a Young 
	diagram with	$p$ boxes we will state Theorem \ref{T2}. 
	To know how the equations work we need some notations and
	definitions from \cite{ROTA}. 
	We will denote by $\pi$ the partitions of the set $\{1,2,\ldots,p\}$ and for any permutation $\sigma\in\Sigma_p$ we denote by
	$par(\sigma)$ the partition obtained by the cycle decomposition of $\sigma$.
	From a partition $\pi$ we can get the corresponding Young diagram and we denote it by
	$shape(\pi)=\lambda=(\lambda_1,\lambda_2,\ldots,\lambda_m)$, where every $\lambda_i$ are the sizes of the blocks 
	of $\pi$ (arranged 
	in non-increasing order).
	
	We need to define a partial order in the set of Young diagrams to give the main definition of the theorem. 
	Let $\lambda=(\lambda_1,\lambda_2,\ldots\lambda_m)$ and $\mu=(\mu_1,\mu_2,\ldots,\mu_r)$ be two Young diagrams (if it is necessary, 
	the last entries of each diagram are filled with zero). Write $\lambda\leq\mu$
	if and only if the following conditions hold:
	\begin{itemize}
		\item $\lambda_1+\lambda_2+\ldots+\lambda_m=\mu_1+\mu_2+\ldots+\mu_r$,
		\item $\lambda_1\leq\mu_1$,
		\item $\lambda_1+\lambda_2\leq\mu_1+\mu_2$,
		\item $\ldots$
	\end{itemize}
	
	According to \cite{ROTA}, the set of minimal elements of the complement of the partially ordered 
	set $\{\mu\,:\,\mu\leq\lambda\}$ is called \textit{critical set}
	of the Young diagram $\lambda$.
	To state Theorem \ref{T2} we also define $Pos(\pi)$ as an element of the
	group algebra $\C\Sigma_p$ given by $\Sigma\sigma$ where the sum is extended to all 
	permutations $\sigma$ such that $par(\sigma)$ is a refinement of $\pi$.
	
			
	
	\begin{bla}\label{T2}(Rota-Stein[4], Theorem at pag. 848)
	If $A\in \Gamma^{\lambda}V\otimes S_{\lambda}V$, and $\pi$ is a partition such that
	$shape(\pi)$ belongs to the critical set of $\lambda$, then
	\begin{equation}\label{poseq}
	\sum_{\sigma\in Pos(\pi)}A(x_{\sigma(1)},\ldots,x_{\sigma(n)})=0
	\end{equation}
	\end{bla}
	
	Each equation stated in the previous theorem is called a \textit{positive equation} and we will denote it
	for brevity as $Pos(\pi)\cdot A=0$.

	We are interested in the product $\underbrace{V\otimes\ldots \otimes V}_{p}$, 
	where the group $\underbrace{GL(V)\times\ldots\times GL(V)}_{p}$ acts in a natural way. Once a basis is fixed in $V$, the tensors 
	can be represented as multidimensional matrices of format $\underbrace{(n)\times\ldots\times (n)}_{p}$.

	A feature of the Segre variety is that it contains a lot of linear subspaces. For any point $x=x_1\otimes\ldots\otimes x_p$,
	the linear space $x_1\otimes\ldots\otimes V_i\otimes\ldots\otimes x_p$ passes through $x$ for $i=1,\ldots,p$. 	Let $x\in 
	X=\Pe^{k_1}\times\ldots\times\Pe^{k_p}=\Pe(V_1)\times\ldots\times \Pe(V_p)$. Then the tangent space at $x$, $T_x X$ 
	is the projectivization of $\oplus_i <x_1>\otimes\ldots \otimes V_i\otimes \ldots\otimes <x_p>$ (see \cite{OTT}). 

	
	

	Now consider $X\subseteq\Pe(V)$ a projective smooth irreducible variety. We say that a hyperplane $H$ is \textit{tangent} to $X$ if
	$H$ contains the tangent space to $X$ at some  point $x\in X$. This means $T_x X\subseteq H$.
	The set of tangent hyperplanes to $X$ has a natural structure of projective irreducible variety, denoted by 
	$X^\vee\subseteq \Pe(V^\vee)$ and called the \textit{dual variety} of $X$.

	Consider the product $X=\underbrace{\Pe^{n-1}\times \ldots\times\Pe^{n-1}}_{p}$ of several projective
		spaces in the Segre embedding into the projective space $\Pe(V^{\otimes ^p})$. 
		The \textit{hyperdeterminant} of format $\underbrace{(n)\times\ldots\times  (n)}_{p}$
		is a homogeneous polynomial	function on $V^{\otimes ^p}$ which is a defining equation of the projectively dual
		variety $X^\vee$. We denote the
		hyperdeterminant by \textit{Det}. For its main properties see \cite{GKZ} and \cite{OTT}.

	
	

	From now on we will refer to a multidimensional matrix $A$ as simply a matrix and we will write $Det(A)$ for its 
	hyperdeterminant (when it exists). The following definition is the main point of the final argument.
	
	 According to \cite{GKZ}, a matrix $A$ is called \textit{degenerate} if there exists  
	 $x^1\otimes\ldots\otimes x^p\in V\otimes\ldots\otimes V$ nonzero such that 
	 $A(x^1,\ldots,\underbrace{V}_{i},\ldots, x^p) = 0\,\,\,\, \forall\, i=1,\ldots,p$.

		By definition the \textit{kernel} is\\
		 $K(A)=\{(x^1\otimes\ldots\otimes x^p)\in V\otimes \ldots \otimes V$ such that 
			$A(x^1,\ldots,V,\ldots,x^p)=0$, $\forall\,i=1,\ldots,p\}$

		
	From these definitions we get, always from \cite{GKZ}, that: 
	
	\begin{prop}
	The following are equivalent:
			\begin{itemize}
				\item $Det(A)=0$
				\item $A$ is degenerate
				\item $K(A)\neq \phi$
			\end{itemize}
	\end{prop} 
	This gives a geometric reformulation of the definition of hyperdeterminant.

\section{Proof of main theorem}

	Using the previous characterization we will get that the hyperdeterminant vanishes if we restrict it to 
	skew-symmetric multidimensional matrices.
	We use the skewness properties:\\if $A\in\bigwedge^{p}V\subseteq V^{\otimes p}$  then $A(x_1,\ldots,x_p)=0$ if $x_i=x_j$
	for some $i,j=1,\ldots,p$.

	\begin{prop} If $p\geq 3$, every $A\in\bigwedge^p V$ is a degenerate matrix, indeed every $x\otimes\ldots\otimes x\in K(A)$
	for all $x\in V$.
	\end{prop}
	
	\begin{proof}
	$A(x,\ldots,x,y,x,\ldots,x)=0$ for every $x,y\in V$ because two elements coincide and $p\geq 3$.
	Hence $x\otimes\ldots\otimes x\in K(A)$ and $A$ is degenerate.
	\end{proof}

	In order to have the same result for the case $A\in \Gamma^\lambda V\otimes S_\lambda V\subseteq V^{\otimes p}$ 
	where $\lambda$ is a 
	Young diagram with $p$ boxes and $\lambda_2\geq 2$ or $\lambda_3\geq 1$  
	we need two lemmas before giving the proof of the main theorem.

	
	
	
	\begin{lem}\label{lem1}
	If $Pos(\pi)\cdot A=0$ where $\pi$ ranges over all partitions of $\{1,2,\ldots,p\}$ such that $shape(\pi)$ is 
	a Young diagram with at least two rows and $p$ boxes, then $A$ is degenerate. 			
	\end{lem} 
	
	\begin{proof}
		
	Let $x,y\in V$. We consider the $p$ unknowns
	$A(y,x,\ldots,x)$,  $A(x,y,x,\ldots,x)$,$\ldots$, \\$A(x,\ldots,x,y)$.
	Since $shape(\pi)=\mu$ has at least two rows, we can consider $\mu_1+1$ of these unknowns 
	and construct a homogeneous linear system.
	
	In the first place, we consider the first $\mu_1+1$ unknowns. 
	By taking the partition constituted by $I=\{2,3,...,\mu_1+1\}$ and its complement
	in $\{1,...,p\}$, and substituting $x_2 = y$ and $x_i = x$ for $i\in I$ with $i\neq 2$, and also $x_i=x$ for $i\notin I$ on the
	positive equations (\ref{poseq}), we get an equation between these unknowns (that correspond with the
	first row of the following matrix $M$ ). By taking the partition constituted by the subset 
	$I = \{1, 2,\ldots, \hat{k},\ldots,\mu_1+1\}$
	with $k \neq 1$ and its complement
	in $\{1,...,p\}$, and substituting $x_1 = y$ and $x_i = x$ for $i\in I$ with $i\neq 1$, and also $x_i=x$ for $i\notin I$ on the
	positive equations (\ref{poseq}), we get $\mu_1$ equations
	between these unknowns (each of them corresponding to the $k$-th row of the following matrix $M$).
	In this way, we get $\mu_1+1$ equations and a homogeneous linear system with the following matrix:
	\begin{equation*}
		M=\left(
			\begin{array}{cccccccc}
			0 		& 1 		& 1 		& 1 		& \ldots 		& 1 		& 1 		& 1\\
			1 		& 0 		& 1 		& 1 		& \ldots 		& 1 		& 1 		& 1 \\
			1 		& 1 		& 0 		& 1 		& \ldots 		& 1 		& 1 		& 1\\
			\vdots  & \vdots    & \vdots	& \vdots	& \vdots		& \vdots	& \vdots	& \vdots\\
			1 		& 1			& 1 		& 1  		& \ldots		& 1			& 0			& 1\\
			1 		& 1			& 1 		& 1  		& \ldots		& 1			& 1			& 0
			\end{array}\right)
	\end{equation*} 
 
 	We can easily see that $M$ is non degenerate. To do this we have to compute the determinant of $M-t\cdot I$.
 	For $t=-1$ we get that $det(M-t\cdot I)=0$. As the rank is $1$, the multiplicity of $-1$ is $\mu_1$, and since the trace is zero,
 	we get that the eigenvalues are just $\underbrace{-1,\ldots,-1}_{\mu_1},(\mu_1)$. Therefore $det(M)=(-1)^{\mu_1}(\mu_1)$
 	and the $\mu_1+1$ unknowns are all zero. 
 
 	We may repeat the same
 	argument for any other subset of $\{\mu_1 + 1\}$ unknowns among the $p$ unknowns.
 	We conclude that $A(y, x, \ldots , x) = A(x, y, x, \ldots , x) =\ldots = A(x, \ldots , x, y) = 0$, so that 
 	$x\otimes\ldots\otimes x\in K(A)$ and $A$ is degenerate.
 		
	
	
	\end{proof} 
	
	\begin{lem}\label{lem2}
	For every partition $\pi$ of $\{1,2,\ldots,p\}$, 
	with $shape(\pi)\neq (p), (p-1,1)$, the critical set of $shape(\pi)$ is composed by Young diagrams
	with at least two rows.
	\end{lem}
	
	\begin{proof}
	By using the partial order explained before we have the following sequence within the set of the Young diagrams with $p$ boxes:
	$$(\underbrace{1,\ldots,1}_{p})\leq (2,\underbrace{1,\ldots,1}_{p-2})\leq (2,2,\underbrace{1,\ldots,1}_{p-4})\leq\ldots$$
	$$\ldots\leq (p-2,2)\leq (p-1,1)\leq (p)$$
	The critical set of each diagram is the next one to the right in this sequence. From the definition of partial order there is 
	no manner in which we can have a diagram with one row inside the sequence. 
	In this context, the only possibilities are $(p-1,1)$ (the critical set is $(p)$)
	and $(p)$ (the critical set is $\emptyset$).   
	\end{proof}
	

	\textbf{Main Theorem.} \textit{When $A\in\Gamma^{\lambda} V\otimes S_{\lambda}V\subseteq V^{\otimes p}$, $|\lambda|=p$ and $p\ge 2$, 
	$Det(A)$ can be nonzero only for $\lambda=(p)$ (corresponding to the  symmetric power ${\mathrm Sym}^p V$) 
	and $\lambda=(p-1,1)$, where $\Gamma^{\lambda}V$ (resp. $S_{\lambda}V$) is the  $\Sigma_p$-module 
	(resp. $GL(V)-$module) associated to $\lambda$.}

	\begin{proof}
	The critical set of $\lambda$ is always a Young diagram with at least two rows by Lemma \ref{lem2}. The proof is complete 
	applying Lemma \ref{lem1}.
	\end{proof}


\begin{thebibliography}{999}

\bibitem{CA} A. Cayley, \textit{On the theory of linear transformations}, Cambridge Math. J. 4 (1845), 1-16, Reprinted in his
Collected Mathematics Papers, Vol. 1, p. 80-94, Cambridge Univ. Press,1889 [microform].

\bibitem{FH} W. Fulton and J. Harris, \textit{Representation theory}, Springer GTM 129.

\bibitem{GKZ} I. M. Gelfand, M. M. Kapranov, and A. V. Zelevinsky, \textit{Discriminants, resultants and multidimensional determinants},
Mathematics: Theory and Applications, Birkhäuser, Boston, MA, 1994. MR 1264417 (95e:14045).

\bibitem{OED} L. Oeding, \textit{The hyperdeterminant of polynomials}, Advances in Math., 231 (3-4), 1308-1326 (2012).

\bibitem{OTT} G. Ottaviani, \textit{An introduction to the hyperdeterminant and to the rank of multidimensional matrices}, 
Commutative Algebra, Expository Papers Dedicates to David Eisenbud on the Occasion of His 65th Birthday (New York) (I. Peeva, ed.),
Springer, 2013, p. 609-638.

\bibitem{PRO} C. Procesi, \textit{Lie groups, an approach through invariants and representations}, 
Universitext, Springer, New York, 2007. 

\bibitem{ROTA} G. Rota and J. A. Stein, \textit{Symmetry classes of functions}, 
Journal of Algebra 171 (1995), p. 845-866.

\bibitem{WEY} J. Weyman, \textit{Cohomology of vector bundles and syzygies}, Cambridge University
Press, 2003.

\end{thebibliography}
\end{document}